\documentclass[12pt]{amsart}
\usepackage{amsmath}

\newtheorem{theorem}{Theorem}  

\newtheorem{proposition}[theorem]{Proposition}
\newtheorem{definition}[theorem]{Definition}
\newtheorem{example}[theorem]{Example}
\newtheorem{corollary}[theorem]{Corollary}
\newtheorem{remar}[theorem]{Remark}
\newtheorem{claim}{Claim}
\renewenvironment{proof}{Proof:\ \ \ }{\QED}
\newenvironment{remark}{\begin{remar}\rm}{\end{remar}}
\newcommand{\QED}{\qed}
\newcommand{\bfind}[1]{\index{#1}{\sl #1}}

\newcommand{\sn}{\par\smallskip\noindent}

\newcommand{\pars}{\par\smallskip}
\newcommand{\parm}{\par\medskip}
\newcommand{\parb}{\par\bigskip}

\newcommand{\N}{\mathbb N}
\newcommand{\Z}{\mathbb Z}
\newcommand{\F}{\mathbb F}
\newcommand{\cP}{\mathcal P}
\newcommand{\cB}{\mathcal B}
\newcommand{\cC}{\mathcal C}
\newcommand{\cal}{\mathcal}
\newcommand{\ui}{i^*}    

\newcommand{\co}{\mbox{\rm co}}
\newcommand{\cl}{\mbox{\rm cl}}
\newcommand{\diam}{\mbox{\rm diam}}
\newcommand{\ovl}[1]{\overline{#1}}

\begin{document}
\title[]{Valuation theory, generalized IFS attractors and fractals}
\author{Jan Dobrowolski
and Franz-Viktor Kuhlmann}
\thanks{The first author was supported by European Union's Horizon 2020 research
and innovation programme under the Marie Sklodowska-Curie grant agreement No 705410. \\
The authors wish to thank the referee for his corrections and suggestions that helped to improve this paper.}
\address{Faculty of Mathematics and Physical Sciences,
University of Leeds,
Leeds LS2 9JT, UK}
\email{J.Dobrowolski@leeds.ac.uk}
\address{Institute of Mathematics, 	
ul. Wielkopolska 15, 	  	  	
70-451 Szczecin, Poland}
\email{fvk@math.usask.ca}
\date{April 10, 2018}
%

\begin{abstract}
Using valuation rings and valued fields as examples, we discuss in
which ways the notions of ``topological IFS attractor'' and ``fractal
space'' can be generalized to cover more general settings.
\end{abstract}
\maketitle

%
%
%
%
%

Given functions $f_1,\ldots,f_n$ on a set $X$, we will
associate to them an \bfind{iterated function system} (IFS), denoted by
\[
F\>=\>[f_1,\ldots,f_n]\>,
\]
where we view $F$ as a function on the power set $\cP(X)$ defined by
\[
{\cal P}(X)\ni S\>\mapsto\>F(S)\,:=\,\bigcup_{i=1}^n f_i(S)\>.
\]
One of the basic approaches to calling a space $X$ ``fractal'' is to ask
that there is an iterated function system $F$ such that $F(X)=X$, and that the
functions in the system satisfy certain additional forms of being ``contracting'':
\begin{definition} \rm
A compact metric space $(X,d)$ is called \bfind{fractal} if there is an iterated function system
$F=[f_1,\ldots,f_n]$ with $F(X)=X$ where the functions $f_i$ are \bfind{weakly contracting}, that is,
$d(f_ix,f_iy)<d(x,y)$ for any distinct $x,y\in X$.
\end{definition}
\sn
Alternatively, one may ask that the functions $f_i$ are \bfind{contracting}, that is, there is some positive real
number $C<1$ such that $d(f_ix,f_iy)\leq C d(x,y)$ for all $x,y\in X$.

Iterated function systems consisting of weakly contracting functions are studied in e.g.\ \cite{dA,dAS,N}.

\pars
In the absence of a metric, one has to find other ways of encoding what is meant by ``contracting''.
In \cite{BaNo}, Banakh and Nowak give a topological analogue for the
common definition of ``fractal'' that uses iterated function systems; for a detailed continuation of this approach, see \cite{BKNNS,BNS}.

\begin{definition} \rm                            \label{top_ifs}
A compact topological space $X$ is called \bfind{fractal} if there is an
iterated function system $F=[f_1,\ldots, f_n]$ consisting of continuous
functions $f_i:X\rightarrow X$ such that $F(X)=X$ and the following
``shrinking condition'' is satisfied:
\sn
(SC) for every open covering ${\cal C}$ of $X$, there is some
$k\in\N$ such that for every sequence $(i_1,\ldots,i_k)\in
\{1,\ldots,n\}^k$ there is $U\in {\cal C}$ with
\[
f_{i_1}\circ\ldots\circ f_{i_k}(X)\>\subset U\>.
\]
\end{definition}

Clearly, it suffices to check (SC) only for finite coverings. If we fix
a basis for the open sets, then it suffices to check (SC) only for
finite coverings consisting of basic sets, as every finite covering can
be refined to such a covering.

\parm
We will now give examples of iterated function systems in the special case where the topology is induced by a
valuation on a field. For general background on valuation theory, see e.g.\ \cite{EP,R1,R2,W}.
For the topology induced by a valuation $v$ on a field $K$ with value group $vK$, one can take
the collection of ultrametric balls
\[
B_\alpha(a)\>:=\> \{b\in K \mid v(a-b)\geq\alpha\} \;\mbox{ where }
\alpha\in vK \mbox{ and } a\in K
\]
as a basis; note that by the ultrametric triangle law, this set is closed
under nonempty intersections over finite subsets. For the same reason, if $v(a-b)\geq\alpha$,
then $B_\alpha(a)=B_\alpha(b)$. The same works when we restrict $v$ to a subring $R$ of $K$, except that then the
values of the elements in $R$ just form a linearly ordered subset of $vK$.

\begin{example}  \rm
We take a prime $p$ and denote by $\F_p$ the finite field with $p$
elements. Then $\F_p$ consists of the elements $\ovl{i}= i+p\Z$, $0\leq i<p$.
We consider the Laurent series ring
\[
R\>:=\>\F_p[[t]]\>=\>\left\{\sum_{j=0}^{\infty} \ovl{ i}_j\, t^j\mid
\ovl{ i}_j\in\F_p\right\}\>.
\]
The $t$-adic valuation $v_t$ on $\F_p[[t]]$ is defined by
\begin{equation}                            \label{t-adic}
v_t\,\sum_{j=0}^{\infty}\ovl{i_j}\, t^j\>=\> \min\{j\mid j\geq 0\mbox{ such that } \ovl{i_j}\ne \ovl{0}\}\>.
\end{equation}

For $0\leq i\leq p-1$, we define a function $f_i$ by
\[
f_i\left(\sum_{j=0}^{\infty}\ovl{ i}_j\, t^j\right)\>:=\>\ovl{i} + \sum_{j=1}^{\infty}
\ovl{i}_{j-1}\, t^j\>=\>\ovl{i}+t\sum_{j=0}^{\infty} \ovl{i_j}\, t^j\>,
\]
where $\ovl{i}$ is understood to be an element of $\F_p\,$. Then
\begin{equation}                         \label{fi}
f_i(R) \>=\> \ovl{i}+tR
\end{equation}
and therefore, the iterated function system
$F=[f_0,\ldots,f_{p-1}]$ satisfies $F(R)=\bigcup_{0\leq i<p} \ovl{i}+tR=R$.
Each ultrametric ball in $R$ (with respect to the $t$-adic valuation of $\F_p[[t]]$) is of the form
\[
B_m\left(\sum_{j=0}^{m-1} \ovl{i_j}\, t^j\right)\>=\> \left\{b\in R\mid t^m\mbox{ divides } b-\sum_{j=0}^{m-1}
\ovl{i_j}\, t^j\right\}
\]
for some integer $m\geq 0$, which we call the radius of the ball.
(The empty sum is understood to be $0$). Given any finite open covering
of $R$ consisting of ultrametric balls, we take $m$ to be the maximum of the radii of all balls in the
covering. Then the covering can be refined to a covering of the form
\[
\left\{\left. B_m\left(\sum_{j=0}^{m-1} \ovl{i_j}\, t^j\right)\,\right|\,
i_0,\ldots,i_{m-1}\in \{0,\ldots,p-1\}\right\}\>.
\]
Choose any $m\geq 1$ and $i_0,\ldots,i_{m-1}\in \{0,\ldots,p-1\}$. By induction on $m$ we derive from (\ref{fi}):
\[
f_{i_0}\circ\ldots\circ f_{i_{m-1}}(R) \>=\> \ovl{i_0}\,+\,\ovl{i_1}\,t\,+\ldots+\,\ovl{i_{m-1}}\, t^{m-1}\,+\,t^m R
 \>=\> B_m\left(\sum_{j=0}^{m-1}\ovl{i_j}\,t^j\right)\,.
\]
Since the functions $f_i$ are continuous in the topology induced by the ultrametric (an argument will be given below
in the more general case of discrete valuation rings), we see that $\F_p[[t]]$ with its ultrametric balls is fractal,
in the sense of Definition~\ref{top_ifs}.
\hfill $\diamond$
\end{example}

Here is an obvious generalization of the previous example.
\begin{example}  \rm
We work in the same situation as in the last example, but now fix an integer $\mu\geq 0$ and for every $\ui
=(\ui(0),\ldots,\ui(\mu))\in \{0,\ldots,p-1\}^{\mu+1}$ we set
\[
f_{\ui}\left(\sum_{j=0}^{\infty}\ovl{i_j}\, t^j\right) \>:=\> \sum_{j=0}^\mu \ovl{\ui(j)}\, t^j \,+\, t^{\mu+1}
\sum_{j=0}^{\infty} \ovl{i_j}\,t^j\>.
\]
Then
\begin{equation}                         \label{fi*}
f_{\ui}(R)\>=\>\sum_{j=0}^\mu \ovl{\ui(j)}\, t^j + t^{\mu+1}R
\end{equation}
and therefore, the iterated function system $F=[f_{\ui}\mid {\ui}\in \{0,\ldots,p-1\}^{\mu+1}]$ satisfies
\[
F(R)\>=\>\sum_{j=0}^{\mu} \F_p t^j + t^{\mu+1}R\>=\> R\>.
\]
Choose any $m\geq 1$ and $\ui_0,\ldots,\ui_{m-1}\in \{0,\ldots,p-1\}^{\mu+1}$.
By induction on $m$ we derive from (\ref{fi*}):
\begin{eqnarray*}
f_{\ui_0}\circ\ldots\circ f_{\ui_{m-1}}(R) &=& \sum_{j=0}^\mu \ovl{\ui_0(j)}\, t^j\,+\,
t^{\mu+1}\sum_{j=0}^\mu \ovl{\ui_1(j)}\, t^j\,+\ldots+\\
&+& (t^{\mu+1})^{m-1}\sum_{j=0}^\mu \ovl{\ui_{m-1}(j)}\, t^j\,+\,(t^{\mu+1})^m R \\
&=& B_{m(\mu+1)}\left(\sum_{k=0}^{m-1} t^{k(\mu+1)}\sum_{j=0}^\mu \ovl{\ui_k (j)} \, t^j \right)\>.
\end{eqnarray*}
\hfill $\diamond$
\end{example}

We generalize our observations to discrete valuation rings (which in
general cannot be presented in power series form, in particular not in
mixed characteristic).

We take a discrete valuation ring $R$ with maximal ideal $M$ and choose a
uniformizing parameter $t\in R$, i.e., the value of $t$ is the smallest positive
element in the value set of $R$. Further, we choose a system of
representatives $S\subset R$ for the residue field $R/M$. Then for every
$s\in S$ we define a function $f_s$ by:
\begin{equation}
f_s(a)\>:=\> s\,+\,ta
\end{equation}
for $a\in R$. Then
\begin{equation}                         \label{fs}
f_s(R) \>=\> s+tR
\end{equation}
and therefore,
\[
\bigcup_{s\in S} f_s(R) \>=\> \bigcup_{s\in S} s+tR\>=\> R\>.
\]
Choose any $m\geq 1$ and $s_0,\ldots,s_{m-1}\in S$. By induction on $m$ we derive from (\ref{fs}):
\[
f_{s_0}\circ\ldots\circ f_{s_{m-1}}(R) \>=\> s_0\,+\,s_1\,t\,+\ldots+\,s_{m-1}\, t^{m-1}\,+\,t^m R \>=\>
B_m\left(\sum_{j=0}^{m-1} s_j\, t^j\right)\>.
\]

If $a,b\in R$ with $a-b\in B_m(0)=t^m R$, then
\[
f_s(a)-f_s(b) \>=\> s+ta - (s+tb) \>=\> t(a-b)\in t^{m+1} R = B_{m+1}(0)\>.
\]
This shows that each $f_s$ is
contracting and hence continuous in the topology induced by the ultrametric.

If $R/M$ is finite, then we have finitely many functions and obtain:
\begin{proposition}
Every discrete valuation ring with finite residue field and equipped
with the canonical ultrametric is fractal (under {\it both} definitions given above).
\end{proposition}

\parm
Note that for any topological space $X$, the existence of a continuous
IFS $F=[f_1,\dots,f_n]$ satisfying conditions $(SC)$ and $X=\bigcup_i
f_i[X]$ from Definition~\ref{top_ifs} implies that $X$ is quasi-compact.
By the following example, it can be seen that these conditions do not
imply that $X$ is Hausdorff (so Definition~\ref{top_ifs} could be also
considered for non-Hausdorff quasicompact spaces):

\begin{example}
Let $X=[0,1]$ be equipped with the topology in which the open sets are $\emptyset$ and the
cofinite sets. Define $f_1,f_2:X\to X$ by $f_1(x)=x/2$,
$f_2(x)=1/2+x/2$. Then the system $(f_1,f_2)$ consists of continuous
functions, and satisfies conditions $(SC)$ and $X=\bigcup_i f_i[X]$.
\hfill $\diamond$
\end{example}

The following definition seems to be the weakest reasonable generalization
of Definition \ref{top_ifs} to possibly infinite function systems.

\begin{definition}          \rm
Let $X$ be a topological space, and $\{f_i:i\in I\}$ any set of
continuous mappings $X\to X$ satisfying $(SC)$, i.e., for any finite
open covering $\cal{U}$ of $X$ there is a natural number $l$ such that
for any $g_1,\dots,g_l\in \{f_i:i\in I\}$, the image $g_1\circ\dots\circ
g_l[X]$ is contained in some $U\in \cal{U}$. We will say that $X$ is a
topological attractor for $\{f_i:i\in I\}$, if $X=\cl(\bigcup_{i\in
I}f_i[X])$. For any cardinal number $\kappa$, we will say that $X$ is a
topological $\kappa$-IFS-attractor, if $X$ is an attractor for some set
of continuous functions satisfying $(SC)$ of cardinality at most $\kappa$.

This definition generalizes Definition \ref{top_ifs} in the sense that being a topological IFS-attractor
is the same as being a compact $n$-IFS-attractor for some $n\in \omega$ (because when $X$
is compact and $I$ is finite, then we have $\cl(\bigcup_{i\in
I}f_i[X])=\bigcup_{i\in
I}f_i[X])$.

The reader may note that our notion of a topological attractor
is not the same as the one used in papers \cite{dAS,N}, where
a nonempty compact set is defined to be a topological attractor if it is
homeomorphic to an attractor.
\end{definition}

For normal spaces, the property of being a $\kappa$-IFS-attractor implies
a bound on the weight (i.e., the minimal cardinality of a basis of the topology):
\begin{proposition}
Suppose $X$ is a normal space which is a $\kappa$-IFS-attractor. Then
$w(X)\leq 2^\kappa +\aleph_0$.
\end{proposition}
\begin{proof}
Choose a system of functions $F=\{f_i:i\in I\}$ of cardinality at most
$\kappa$ satisfying $(SC)$ such that $X$ is an attractor for $F$, i.e.,
$X=\cl(\bigcup_{i\in I}f_i[X])$.
\begin{claim}
For any natural number $l$, we have that
\begin{equation}                            \label{X=}
X\>=\>\cl\left(\bigcup_{g_1,\dots,g_l\in F} g_1\circ\dots\circ
g_l[X]\right)\>.
\end{equation}
\end{claim}
{\em Proof of the claim.}
We proceed by induction on $l$. Suppose that (\ref{X=}) holds.
Then for every $i\in I$, we get by the continuity of $f_i$ that
\begin{eqnarray*}
f_i[X] & \subset & \cl\left(f_i\left[\bigcup_{g_1,\dots,g_l\in F}
g_1\circ\dots\circ g_l[X]\right]\right) \\
& = & \cl\left(\bigcup_{g_1,\dots,g_l\in F}f_i\circ g_1\circ\dots\circ
g_l[X]\right) \\
& \subset & \cl\left(\bigcup_{g_1,\dots,g_{l+1}\in F}g_1\circ\dots\circ
g_{l+1}[X]\right)\>.
\end{eqnarray*}
Thus, we obtain that $X= \cl(\bigcup_{i\in I}f_i[X])\subset
\cl (\bigcup_{g_1,\dots,g_{l+1}\in F}g_1\circ\dots\circ g_{l+1}[X])$.
This completes the proof of the claim.

\parm
Define
\[
\mathcal{B} =\left\{\left. X\setminus\cl\left(\bigcup_{(g_1,\dots,g_l)
\in I} g_1\circ\dots\circ g_l[X]\right) \,\right |\, l<\omega,
I\subset F^l\right\}\>.
\]
Clearly, $|\mathcal{B}|\leq \sum_{l<\omega} |P(F^l)|= 2^\kappa\aleph_0=2^\kappa+
\aleph_0$. We will show that
$\mathcal{B}$ is a basis of $X$. So take any open subset $U$ of $X$ and
$x\in U$. Since $X$ is normal, we can choose open sets $V_1,V_2$ such
that
\[
x\in V_1\subset \cl(V_1)\subset V_2\subset \cl(V_2)\subset U\>.
\]
Let $l$ be as in the condition $(SC)$ for $F$ and the covering
$\{V_2,X\setminus \cl(V_1)\}$ of $X$. Define $J=\{(g_1,\dots,g_l)\in
F^l:g_1\circ\dots\circ g_l[X]\subset X\setminus \cl(V_1)\}$ and
\sn
$W=X\setminus \cl(\bigcup_{(g_1,\dots,g_l) \in J}g_1\circ\dots\circ
g_l[X])\in \mathcal{B}$. Since $V_1$ is disjoint from
$\bigcup_{(g_1,\dots,g_l) \in J}g_1\circ\dots\circ g_l[X]$, we get that
$x\in W$. It remains to check that $W\subset U$. Take any $y \in
X\setminus U$. We will show that $y\in \cl(\bigcup_{(g_1,\dots,g_l) \in
J}g_1\circ\dots\circ g_l[X])$. Take any open neighbourhood $Z$ of $y$.
By the claim, $g_1\circ\dots\circ g_l[X]$ meets $Z\cap (X\setminus
\cl(V_2))$ for some $h_1,\dots,h_l\in F$. Then the image
$g_1\circ\dots\circ g_l[X]$ is not contained in $V_2$, so it is
contained in $X\setminus \cl(V_1)$ and $(h_1,\dots,h_l)\in J$. Therefore, $Z$
meets $\bigcup_{(g_1,\dots,g_l) \in J}g_1\circ\dots\circ g_l[X]$, and we are done.
\end{proof}

\pars
The above proposition applies in particular to compact spaces (which are
known to be normal). In particular, we obtain that every topological
IFS-attractor has a countable basis. Thus, by the Urysohn metrization
theorem, we get:

\begin{corollary}\label{metrizable}
Every topological IFS-attractor is metrizable.
\end{corollary}

Condition $(SC)$ is not satisfied in some natural examples where the
metric shrinking condition is satisfied (i.e., $\lim_l\sup_{i_1,\dots,i_l}
\diam(f_{i_1}\circ\dots\circ f_{i_l}[X])=0$):
\begin{example}\label{baire}
Let $X=\omega^{\omega}$ be the Baire space (which is homeomorphic to
$k((t))$ considered with the valuation topology, where $k$ is any field
of cardinality $\aleph_0$). For any $i<\omega$, define $f_i:X\to X$ as
follows: $f_i(x)(0)=i$ and $f_i(x)(n)=x(n-1)$ for $n>0$. Then $(SC)$ is
not satisfied for $f_i,i<\omega$, which is witnessed by the covering
$\{U,X\backslash U\}$, where $U=\bigcup_{n<\omega}\{x\in
X:x(0)=n,x(1)=\dots =x(n)=0\}$. \hfill $\diamond$
\end{example}

Thus, we want to consider another topological shrinking condition, in
which we are allowed to choose a basis from which the covering sets are
taken. However, to make it possible to cover in this way the whole space
(which is not assumed to be compact), we allow one of the covering sets
to be not in the fixed basis. This leads to the following definition:

\begin{definition}
A family of functions $(f_i)_{i\in I}$ on a topological space $X$
satisfies $(SC*)$ if there is a basis $\mathcal{B}$ of $X$ such that for
every finite open covering ${\cal C}$ of $X$ containing at most one set
which is not in $\mathcal{B}$, there is some $k\in\N$ such that, for
every sequence $(i_1,\ldots,i_k)\in I^k$, there is $U\in {\cal C}$ with
\[
f_{i_1}\circ\ldots\circ f_{i_k}(X)\>\subset U\>.
\]
\end{definition}

Every space is an attractor for the set of all constant functions from
$X$ to $X$ (i.e., is covered by their images). We will say that $X$ is
a weak $*$-IFS attractor if it is an attractor for a set of functions
satisfying $(SC*)$ of a cardinality smaller than $|X|$. We will say that
$X$ is a $*$-IFS attractor if it is an attractor for a finite set of
functions satisfying $(SC*)$.

Clearly, we have:
\begin{remark}
If $X$ is a compact space, then it is a $*$-IFS attractor if and only
if it is a topological IFS attractor.
\end{remark}

By the following example, it can be seen that being a $*$-IFS attractor
does not imply compactness:
\begin{example}
Let $X=\omega$ be considered with the discrete topology. Define
$f_0,f_1:X\to X$ by $f_0(n)=0$ and $f_1(n)=n+1$. Then $X$ is a $\,*$-IFS
attractor for $\{f_0,f_1\}$, so $X$ is a $*$-IFS attractor.
\end{example}
\begin{proof}
We choose a basis $\mathcal{B}$ consisting of all singletons. Consider
any covering of $X$ of a form $\{U,\{n_1\},\dots,\{n_l\}\}$. Then it is
sufficient to take $k=\max(n_1,\dots,n_l)+1$.
\end{proof}

\begin{example}
Let $X=\omega^{\omega}$, and let $f_i:X\to X, i<\omega$ be as in
Example~\ref{baire}. Then, $(f_i)_{i<\omega}$ satisfies $(SC*)$, so $X$
is a weak $*$-IFS attractor. More generally, for any cardinal number
$\kappa$, the space $\kappa ^\omega$ is an attractor for a set of
functions of cardinality $\kappa$, so it is a weak $*$-IFS attractor if
$\kappa<\kappa^{\omega}$ (this holds for example for all cardinals with
countable cofinality, so for unboundedly many cardinals).
\end{example}
\begin{proof}
For any $\alpha \in \kappa$, define $f_{\alpha}:\kappa^{\omega}\to
\kappa^{\omega}$ by $f_{\alpha}(x)(0)=\alpha$ and
$f_{\alpha}(x)(n)=x(n-1)$ for $n>0$. We choose the standard basis of
$\kappa^{\omega}$, i.e.,
\[
\mathcal{B}=\{A_x:x\in \kappa^k, k<\omega\},
\]
where $A_x=\{y\in \kappa^{\omega}: x\subset y\}$. Write $|x|=k$ if $x\in \kappa^k$.
Choose any open covering of $\kappa^{\omega}$ of the form
$\{U,A_{x_1},\dots,A_{x_n}\}$. Put $k=\max(|x_1|,\dots,|x_n|)$. For any
sequence $(\alpha_0,\dots,\alpha_{k-1})\in \kappa^k$ we have that
$f_{\alpha_0}\circ\dots\circ f_{\alpha_{k-1}}[\kappa^{\omega}]=A_y$,
where $y(i)=\alpha_i$ for all $i<k$, so this image is either contained
in one of the sets $A_{x_1},\dots,A_{x_n}$, or disjoint from all of them
and thus contained in $U$.
\end{proof}

\begin{proposition}\label{ordered}
Suppose $A$ is a densely ordered abelian group and $|\cdot |:A\to \{a\in
A\mid a\geq0\}$ is the associated absolute value. Consider a collection of
functions $f_i:A\to A,i\in I$. Suppose that there is a sequence
$(a_i)_{i<\omega}$ of positive elements of $A$ which converges to 0,
and that for every $k$ and any sequence $(i_1,\ldots,i_k)\in I^k$ we have
$\diam(f_{i_1}\circ\dots\circ f_{i_k}[A])<a_i$. Then $f_i:A\to A,i\in I$
satisfies $SC*$ (where we consider $A$ with the order topology).
\end{proposition}
\begin{proof}
We choose a basis $\mathcal{B}$ of the order topology on
$A$ consisting of all open intervals. We consider any
covering of $A$ of the form $\{U,(a_1,b_1),\dots,(a_n,b_n)\}$. For any
$i$ there is $c_i>0$ such that each of the intervals
$(a_i-c_i,a_i+c_i)$, $(b_i-c_i,b_i+c_i)$ is contained in one of the sets
from the covering. Now, choose $k$ such that for
every sequence $(i_1,\ldots,i_k)\in I^k$ we have
$\diam(f_{i_1}\circ\dots\circ f_{i_k}[A])< c:=1/2\min(c_1,\dots,c_n,b_1-a_1,\dots,b_n-a_n)$.
Then for $a\in A$,, $[a-c,a+c]$ is a subset of some of the sets from the covering:
otherwise, by the choice of $c_i$'s,  $a$ would be at distance $>c$ from all $a_i$'s and
$b_i$'s, and hence a could not belong any of the intervals $(a_i,b_i)$ (as in that case we
would have $[a-c,a+c]\subset (a_i,b_i)$.
But that would mean that $[a-c,a+c]\subset U$.

By the above, any set of diameter smaller than $c$ is contained in one of the
sets from the covering, so we are done.
\end{proof}

\begin{corollary}
$\mathbb{R}$ is a weak $*$-IFS attractor.
\end{corollary}
\begin{proof}
Take a continuous bijection $f_0:\mathbb{R}\to (-1,1)$ which is
Lipschitz with constant $1/2$. Define $f_n(x)=n+f_0(x)$ for any integer
$n$. Then, clearly, the family $\{f_{n}:n\in \mathbb{Z}\}$ satisfies the
assumptions of Proposition \ref{ordered}, so we obtain that it satisfies
$SC*$. Of course, $\mathbb{R}$ is an attractor for that family (and has
a bigger cardinality).
\end{proof}

\parb
A fractal space is compact. If we have a space that is only locally
compact, one can ask whether it is ``locally fractal'', that is, whether
every element is contained in a fractal subspace.

\begin{example}  \rm
We consider the Laurent series field
\[
K\>:=\>\F_p((t))\>=\>\left\{\sum_{j=\ell}^{\infty} \ovl{i_j}\, t^j\mid
\ell\in\Z\,,\, \ovl{i_j} \in\F_p\right\}\>.
\]
The $t$-adic valuation $v_t$ on $\F_p((t))$ is defined by
\[
v_t\,\sum_{j=\ell}^{\infty}\ovl{i_j}\, t^j\>=\> \ell \quad\mbox{ if \ } \ovl{i_\ell}\ne\ovl{0}\>.
\]

For every $k\in\Z$, the function $t^k\F_p((t))\ni c\mapsto t^{\ell-k}c\in
t^\ell\F_p((t))$ is a homeomorphism w.r.t.\ the topology induced by the $t$-adic valuation.
On the other hand, $K=\bigcup_{k\in\Z}
t^k\F_p((t))$. So we see that $\F_p((t))$ is the union over an increasing chain of
mutually homeomorphic fractal spaces.
\hfill $\diamond$
\end{example}

However, we wish to show that $\F_p((t))$ is locally fractal in a stronger sense. The idea is to write $\F_p((t))$
as a union over a collection of mutually homeomorphic fractal subspaces and extend the
functions we have used for $B_0(0)=\F_p[[t]]$ in a suitable way so that they work simultaneously for all of these
subspaces. To this end, we observe that for any two $a,b\in K$, the function $B_0(a)\ni c\mapsto c-a+b\in
B_0(b)$ is a homeomorphism. Note that for each $\ell<0$
there are only finitely many elements in $B_\ell(0)=t^\ell \F_p[[t]]$ that are
non-equivalent modulo $B_0(0)$, so we can write $B_\ell(0)$ as a finite union of the form $\bigcup_j B_0(a_j)$.

\begin{example}  \rm
We extend the functions $f_i$ we used for $R=\F_p[[t]]$ by setting:
\[
f_i\left(\sum_{j=\ell}^{\infty} \ovl{i_j}\, t^j\right) \>:=\> \sum_{j=\ell}^{-1} \ovl{i_j}\, t^j \,+ {\ovl i} +
\sum_{j=1}^{\infty} \ovl{i_{j-1}}\, t^j \>=\> \sum_{j=\ell}^{-1} \ovl{i_j}\, t^j \,+ \ovl{i} +t\sum_{j=0}^{\infty}
\ovl{i_j}\, t^j\>.
\]
For every $a=\sum_{j=\ell}^{\infty} s_j\, t^j\in\F_p((t))$ with $s_j\in\F_p$ we have that
\[
B_0(a) \>=\> a+R \>=\> \sum_{j=\ell}^{-1} s_j\, t^j +R\>.
\]
Hence,
\begin{equation}                         \label{filoc}
f_i(B_0(a))\>=\>f_i\left(\sum_{j=\ell}^{-1} s_j\, t^j + R\right)
\>=\>\sum_{j=\ell}^{-1} s_j\, t^j + \ovl{i} +tR
\end{equation}
and therefore,
\[
\bigcup_{i=0}^{p-1} f_i(B_0(a))\>=\>\sum_{j=\ell}^{-1} s_j\, t^j +
\bigcup_{i=0}^{p-1} (\ovl{i} +tR) \>=\> \sum_{j=\ell}^{-1} s_j\, t^j +R \>=\>B_0(a)\>.
\]
Choose any $m\geq 1$ and $i_0,\ldots,i_{m-1}\in \{0,\ldots,p-1\}$. By induction on $m$ we derive from
(\ref{filoc}):
\begin{eqnarray*}
f_{i_0}\circ\ldots\circ f_{i_{m-1}}(B_0(a)) &=& \sum_{j=\ell}^{-1} s_j\, t^j + \ovl{i_0}\,+\,\ovl{i_1}\,t\,
+\ldots+\,\ovl{i_{m-1}}\, t^{m-1}\,+\, t^m R \\
&=& B_m \left(\sum_{j=\ell}^{-1} s_j\, t^j + \sum_{j=0}^{m-1} \ovl{i_j}\, t^j\right) \>\subset\> B_0(a)\>.
\end{eqnarray*}
\hfill $\diamond$
\end{example}

\pars
For arbitrary discretely valued fields $(K,v)$ with valuation ring $R$ and valuation ideal $M$ we can proceed as
follows. As before, we choose a uniformizing parameter $t\in K$ and a system of representatives $S\subset R$ for the
residue field $Kv$. We set
\[
K^- \>:=\> \left\{\left.\sum_{j=\ell}^{-1} s_j\, t^j \,\right|\> 0\geq\ell\in\Z \mbox{ and } s_\ell,\ldots,s_{-1}\in S
\right\}\>.
\]
Then for every $a\in R$ there is a unique element $a^-\in K^-$ such that $a-a^-\in R$. For every $s\in
S$ we define a function $f_s: K\rightarrow K$ by
\[
f_s(a) \>:=\> a^-\,+\,s\,+\,t(a-a^-)\>.
\]
For every $a\in K$ we obtain that
\begin{equation}                         \label{fsloc}
f_s(B_0(a)) \>=\> a^- + s + tR
\end{equation}
and therefore,
\[
\bigcup_{s\in S} f_s(B_0(a))\>=\> a^- + \bigcup_{s\in S} (s + tR) \>=\>a^- +R \>=\>B_0(a)\>.
\]
Choose any $m\geq 1$ and $s_0,\ldots,s_{m-1}\in S$. By induction on $m$ we derive from
(\ref{fsloc}):
\begin{eqnarray*}
f_{s_0}\circ\ldots\circ f_{s_{m-1}}(B_0(a)) &=& \sum_{j=\ell}^{-1} a_j\, t^j + s_0\,+\,s_1\,t\,
+\ldots+\, s_{m-1}\, t^{m-1}\,+\, t^m R \\
&=& B_m \left(a^- + \sum_{j=0}^{m-1} s_j\, t^j\right) \>\subset\> B_0(a)\>.
\end{eqnarray*}

Take $b,c\in B_0(a)$. Then $b-c\in B_0(0)$, and if $b-c\in B_m(0)$ with $m\geq 0$, then $b^-=c^-$ and
\[
f_s(b)-f_s(c) \>=\> b^- + s + t(b-b^-) - (c^- + s + t(c-c^-)) \>=\> t(b-c)\in B_{m+1}(0)\>.
\]
This shows that each $f_s$ is contracting and hence continuous in the topology induced by the ultrametric.

\parm
We define:
\begin{definition} \rm
A locally compact metric space $(X,d)$ is \bfind{locally fractal} if it is the union over a collection of mutually
homeomorphic subspaces $X_j\,$, $j\in J$, and there is a system $F=[f_1,\ldots,f_n]$ of functions $f_i:X\rightarrow
X$ such that for every $j\in J$, $X_j$ is fractal w.r.t.\ the restrictions of the functions $f_i$ to $X_j\,$.
\end{definition}

\begin{definition} \rm
A locally compact topological space $X$ is \bfind{locally fractal} if it is the union over a collection of mutually
homeomorphic subspaces $X_j\,$, $j\in J$, and there is a system $F=[f_1,\ldots,f_n]$ of functions $f_i:X\rightarrow
X$ such that for every $j\in J$, $X_j$ is (topologically) fractal w.r.t.\ the restrictions of the functions $f_i$ to $X_j\,$.
\end{definition}

Note that we do not require the functions $f_i$ to be continuous or contracting or to satisfy (SC) {\it on all of}
$X$. Indeed, the functions we constructed above have the property that if $(a-b)^-\ne 0$, then
$(f_i(a)-f_i(b))^-=(a-b)^-$.

We have proved:
\begin{proposition}
Every discretely valued field with finite residue field is locally fractal under both definitions.
\end{proposition}


\begin{thebibliography}{99}

\bibitem{BaNo} T.~Banakh and M.~Nowak:
{\em A 1-dimensional Peano continuum which is not an IFS attractor},
Proc.\ Amer.\ Math.\ Soc.\ {\bf 141} (2013), 931--935

\bibitem{BKNNS} T.~Banakh, W.~Kubi\'s, N.~Novosad, M.~Nowak and F.~Strobin: {\em Contractive function systems,
their attractors and metrization}, Topological Methods in Nonlinear Analysis {\bf 46} (2015), 1029--1066

\bibitem{BNS} T.~Banakh, M.~Nowak and F.~Strobin: {\it Detecting topological and Banach fractals among zero-
dimensional spaces}, Topology Appl.\ {\bf 196} (2015), part A, 22--30

\bibitem{dA} E.~D'Aniello: {\it Non-self-similar sets in $[0,1]^N$ of arbitrary dimension}, J.\ Math.\ Anal.\ Appl.\
{\bf 456} (2017), 1123--1128
\bibitem{dAS} E.~D'Aniello and T.H.~Steele: {\it Attractors for iterated function systems}, J.\ Fractal Geom.\ {\bf 3}
(2016), 95--117

\bibitem{EP} A.J.~Engler and A.~Prestel: {\it Valued fields}, Springer Monographs in Mathematics, Springer-Verlag,
Berlin, 2005

\bibitem{N} M.~Nowak: {\it Topological classification of scattered IFS-attractors}, Topology Appl.\ {\bf 160} (2013),
1889--1901

\bibitem{R1} P.~Ribenboim: {\it Th\'eorie des valuations}, Les
Presses de l'Uni\-versit\'e de Montr\'eal, Montr\'eal, 2nd ed.\ 1968
\bibitem{R2} P.~Ribenboim: {\it The Theory of Classical Valuations}, Springer Monographs in Mathematics 1999

\bibitem{W} S.~Warner: {\it Topological fields}, Mathematics studies {\bf 157}, North Holland, Amsterdam, 1989
\end{thebibliography}
\end{document}